\documentclass[leqno,a4paper,12pt]{amsart}

\usepackage{amscd,amssymb,amsmath,amsfonts,amsthm, mathrsfs,
  graphicx,verbatim,hyperref,microtype,enumerate}
\usepackage[OT1,T1]{fontenc}
\usepackage[all]{xypic}
\usepackage{pdfpages}
\usepackage{tikz}
\usepackage[all]{xy}
\usepackage{tabularx}

\usepackage{color}

\newtheorem{theorem}{Theorem}[section]
\newtheorem{corollary}[theorem]{Corollary}

\newtheorem{proposition}[theorem]{Proposition}
\theoremstyle{definition}
\newtheorem{remark}[theorem]{Remark}
\newtheorem{definition}[theorem]{Definition}
\newtheorem{example}[theorem]{Example}

\numberwithin{equation}{section}
\newcommand{\abs}[1]{\left|#1\right|}

\def\CA{{\mathscr A}}
\def\CB{{\mathscr B}}

\def\tee{t}
\def\eye{i}

\newcommand{\cd}[2][]{\vcenter{\hbox{\xymatrix#1{#2}}}}

\newcommand{\ltwocell    }[3][0.5]{\ar@{}[#2] \ar@{=>}?(#1)+/r 0.2cm/;?(#1)+/l 0.2cm/^{#3}}
\newcommand{\rtwocell    }[3][0.5]{\ar@{}[#2] \ar@{=>}?(#1)+/l 0.2cm/;?(#1)+/r 0.2cm/^{#3}}
\newcommand{\dtwocell    }[3][0.5]{\ar@{}[#2] \ar@{=>}?(#1)+/u 0.2cm/;?(#1)+/d 0.2cm/^{#3}}
\newcommand{\utwocell    }[3][0.5]{\ar@{}[#2] \ar@{=>}?(#1)+/d 0.2cm/;?(#1)+/u 0.2cm/_{#3}}

\newcommand{\dltwocell   }[3][0.5]{\ar@{}[#2] \ar@{=>}?(#1)+/ur 0.2cm/;?(#1)+/dl 0.2cm/^{#3}}
\newcommand{\drtwocell   }[3][0.5]{\ar@{}[#2] \ar@{=>}?(#1)+/ul 0.2cm/;?(#1)+/dr 0.2cm/^{#3}}
\newcommand{\ultwocell   }[3][0.5]{\ar@{}[#2] \ar@{=>}?(#1)+/dr 0.2cm/;?(#1)+/ul 0.2cm/^{#3}}
\newcommand{\urtwocell   }[3][0.5]{\ar@{}[#2] \ar@{=>}?(#1)+/dl 0.2cm/;?(#1)+/ur 0.2cm/^{#3}}

\newcommand{\twosimp}[7]{
       \cd{
         & {#2} \ar[dr]^{#5} \ar@{}[d]|(0.6){#7} \\
        {#1} \ar[rr]_{#6} \ar[ur]^{#4} & & {#3} }
}

\newcommand{\monnerve}{
 \cd[@C-1em]{
    (\xcd \otimes \xbc) \otimes \xab \ar[rr]^-{\alpha} \ar[d]_{\xbcd \otimes 1} & &
    \xcd \otimes (\xbc \otimes \xab) \ar[d]^{1 \otimes \xabc} \\
    \xbd \otimes \xab \ar[r]_-{\xabd} & \xad & \xcd
    \otimes \xac \ar[l]^-{\xacd}
  }
}

\newcommand{\laxnerve}{
 \cd[@C-1em]{
    (\xcd \times \xbc) \times \xab \ar[rr]^-{\cong}
    \rtwocell{drr}{\xabcd} \ar[d]_{\xbcd \times 1} & &
    \xcd \times (\xbc \times \xab) \ar[d]^{1 \times \xabc} \\
    \xbd \times \xab \ar[r]_-{\xabd} & \xad & \xcd
    \times \xac \ar[l]^-{\xacd}
  }
}

\newcommand{\pent}[1][0.1cm]{
\xybox{<#1,0cm>:
    \POS (0,15)*+{\xc}="0", 
         (-14,5)*+{\xa}="1", 
         (-9,-12)*+{\xb}="2", 
         (9,-12)*+{\xd}="3", 
         (14,5)*+{\xe}="4",
         (0,0.5)*+{\xo}="C"
    \POS"1" \ar "0"^{\labelstyle \xac}|{}="01"
    \POS"1" \ar "2"_{\labelstyle \xab}|{}="12"
    \POS"2" \ar "3"_{\labelstyle \xbd}|{}="23"
    \POS"3" \ar "4"_{\labelstyle \xde}|{}="34"
    \POS"0" \ar "4"^{\labelstyle \xce}|{}="04"
}
}

\begin{document}
%\leftmargini=2em

\title{The Catalan simplicial set}

\author[Buckley]{Mitchell Buckley}
\address{Department of Computing, Macquarie University NSW 2109, Australia}
\email{mitchell.buckley@mq.edu.au}

\author[Garner]{Richard Garner}
\address{Department of Mathematics, Macquarie University NSW 2109, Australia}
\email{richard.garner@mq.edu.au}

\author[Lack]{Stephen Lack}
\address{Department of Mathematics, Macquarie University NSW 2109, Australia}
\email{steve.lack@mq.edu.au}

\author[Street]{Ross Street}
\address{Department of Mathematics, Macquarie University NSW 2109, Australia}
\email{ross.street@mq.edu.au}

% \author{Mitchell Buckley, Richard Garner, Stephen Lack and Ross Street}

%\date{\small{1 July 2013}}
% \address{
% Centre of Australian Category Theory,
% Macquarie University NSW 2109,
% Australia}

% \email{mitchell.buckley@mq.edu.au, richard.garner@mq.edu.au, steve.lack@mq.edu.au, ross.street@mq.edu.au}

\date{\today}

%\copyrightyear{2013}

\thanks{The authors gratefully acknowledge the support of
Australian Research Council Discovery Grant DP130101969,
Buckley's Macquarie University Postgraduate Scholarship,
Garner's Australian Research Council Discovery Grant DP110102360, and
Lack's Australian Research Council Future Fellowship. }

\maketitle
%\tableofcontents

\section{Introduction}

The $n$th Catalan number $C_n$, given explicitly by
$\frac{1}{n+1}\binom{2n}{n}$, is well-known to be the answer to many
different counting problems; for example, it is the number of
bracketings of an $(n+1)$-fold product.  Thus there are many $\mathbb
N$-indexed families of sets whose cardinalities are the Catalan
numbers; Stanley~\cite{StanleyEC2, StanleyAddendum} describes at
least 205 such.

A Catalan family of sets may bear extra structure that is invisible in
the mere sequence of Catalan numbers. For example, one presentation of
the $n$th Catalan set is as the set of functions $f \colon \{1, \dots,
n\} \to \{1, \dots, n\}$ which preserve order and satisfy $f(k)
\leqslant k$ for each $k$. The set of such functions is a monoid under
composition, and in this way we obtain the \emph{Catalan
  monoids}~\cite{Sol} which are of importance to combinatorial
semigroup theory. For another example, a result due to
Tamari~\cite{Tamari1962The-algebra} makes each Catalan set into a
lattice, whose ordering is most clearly understood in terms of
bracketings of words, as the order generated by the basic inequality
$(xy)z \leqslant x(yz)$ under substitution.

The first main objective of this paper is to describe another kind of
structure borne by Catalan families of sets.  We shall show how to
define functions between them in such a way as to produce a simplicial
set $\mathbb C$, which is the ``Catalan simplicial set'' of the title.
The simplicial structure can be defined in various ways, but the most
elegant makes use of what seems to be a new presentation of the
Catalan sets that relies heavily on the Boolean algebra $\mathbf 2$.

Simplicial sets are abstract, combinatorial entities, most often used
as models of spaces in homotopy theory, but flexible enough to also
serve as models of higher
categories~\cite{Lurie2009Higher,Verity2008Complicial}. Therefore, we
might hope that the Catalan simplicial set had some natural role to
play in homotopy theory or higher category theory.  Our second
objective in this paper is to affirm this hope, by showing that
the Catalan simplicial set has a \emph{classifying property} with
respect to certain kinds of categorical structure. More precisely, we
will consider simplicial maps from $\mathbb C$ into the nerves of
various kinds of higher category (the \emph{nerve} of such a structure
is a simplicial set which encodes its cellular data). We will see that:
\begin{enumerate}[(a)]
\item Maps from $\mathbb C$ to the nerve of a monoidal category $\mathscr V$ are the same thing
  as monoids in $\mathscr V$;
\item Maps from $\mathbb C$ to the nerve of a bicategory $\mathscr B$ are the same thing
  as monads in $\mathscr B$;
\item Maps from $\mathbb C$ to the \emph{pseudo} nerve of the monoidal
  bicategory $\mathrm{Cat}$ of categories and functors are the same
  thing as monoidal categories;
\item Maps from $\mathbb C$ to the \emph{lax} nerve of the monoidal
   bicategory $\mathrm{Cat}$ are the same thing as \emph{skew-monoidal
     categories}.
 \end{enumerate}

 Skew-monoidal categories generalise Mac~Lane's notion of monoidal
 category~\cite{ML1963} by dropping the requirement of invertibility
 of the associativity and unit constraints; they were introduced
 recently by Szlach\'anyi~\cite{Szl2012} in his study of bialgebroids,
 which are themselves an extension of the notion of quantum group. The
 result in (d) can be seen as a coherence result for the notion of
 skew-monoidal category, providing an abstract justification for the
 axioms. Thus the work presented here lies at the interface of several
 mathematical disciplines:
\begin{enumerate}[\ \ \textbullet\,]
\item combinatorics, in the form of the Catalan numbers;
\item algebraic topology, via simplicial sets and nerves;
\item quantum groups, through recent work on bialgebroids;
\item logic, through the distinguished role of the Boolean algebra $\mathbf{2}$; and
\item category theory.
\end{enumerate}

Nor is this the end of the story. Monoidal categories and
skew-monoidal categories can be generalised to notions of
\emph{monoidale} and \emph{skew monoidale} in a monoidal bicategory;
this has further relevance for quantum algebra, since Lack and Street
showed in~\cite{Smswqc} that quantum categories in the sense
of~\cite{QCat} can be described using skew monoidales. In a sequel to
this paper, we will generalise (c) and (d)  to prove that:
\begin{enumerate}[(a)]
\addtocounter{enumi}{4}
\item Maps from $\mathbb C$ to the pseudo nerve of a monoidal
  bicategory $\mathscr W$ are the same thing as monoidales in
  $\mathscr W$; and
\item Maps from $\mathbb C$ to the lax nerve of a monoidal bicategory
  $\mathscr W$ are the same thing as skew monoidales in $\mathscr W$.
\end{enumerate}
The results (a)--(f) use only the lower dimensions of the Catalan
simplicial set, and we expect that its higher dimensions in fact
encode \emph{all} of the coherence that a higher-dimensional monoidal
object should satisfy. We therefore hope also to show that:
\begin{enumerate}[(a)]
\addtocounter{enumi}{6}
\item Maps from $\mathbb C$ to the pseudo nerve of the monoidal
  tricategory $\mathrm{Bicat}$ of bicategories are the same thing as
  monoidal bicategories;
\item Maps from $\mathbb C$ to the homotopy-coherent nerve of the
  monoidal simplicial category $\infty\text-\mathrm{Cat}$ of
  $\infty$-categories are the same thing as monoidal
  $\infty$-categories in the sense of~\cite{LurieHA};
\end{enumerate}
together with appropriate skew analogues of these results.

Finally, a note on the genesis of this work. We have chosen to present
the Catalan simplicial set as basic, and its classifying properties as
derived. This belies the method of its discovery, which was to look
for a simplicial set with the classifying property (d); the link with
the Catalan numbers only later came to light. The notion that a
classifying object as in (d) might exist is based on an old idea of
Michael Johnson's on how to capture not only associativity but also
unitality constraints simplicially.  He reminded us of this in a
recent talk~\cite{MJ2013} to the Australian Category Seminar.

\section{The Catalan simplicial set}
\label{sec:catal-simpl-set}
In this section we define and investigate the Catalan simplicial set.
We begin by recalling some basic definitions. We write $\Delta$ for
the \emph{simplicial category}, whose objects are non-empty finite
ordinals $[n] = \{0,\dots,n\}$ and whose morphisms are
order-preserving functions, and write $\mathrm{SSet}$ % or
% $[\Delta^{\mathrm{op}},\mathrm{Set}]$
for the category of presheaves
on $\Delta$. Objects $X$ of $\mathrm{SSet}$ are called {\em simplicial
  sets}; we think of them as glueings-together of discs, with the
$n$-dimensional discs in that glueing labelled by the set $X_n :=
X([n])$ of \emph{$n$-simplices} of $X$. We write $\delta_i \colon
[n-1]\to [n]$ and $\sigma_{i} \colon [n+1]\to [n]$ for the maps of
$\Delta$ defined by
\begin{equation*}
  \delta_i(x) = \begin{cases} x & \text{if } x < i \\ x+1 & \text{otherwise}
  \end{cases} \quad \text{and} \quad \sigma_i(x) = \begin{cases} x &
    \text{if } x \leqslant i \\ x-1 & \text{otherwise.}
  \end{cases}
\end{equation*}
The action of these morphisms on a simplicial set $X$ yields functions
$d_{i} \colon X_n \to X_{n-1}$ and $s_i \colon X_n \to X_{n+1}$, which
we call \emph{face} and \emph{degeneracy} maps. An $(n+1)$-simplex $x$
is called {\em degenerate} when it is in the image of some $s_i$, and
\emph{non-degenerate} otherwise. The face and degeneracy maps of a
simplicial set satisfy the following \emph{simplicial identities}:
\begin{equation*}
\begin{aligned}
d_i d_j &= d_{j-1} d_i \qquad \!\text{for $i < j$} \\
s_i s_j &= s_{j+1} s_i \qquad \text{for $i \leqslant j$}
\end{aligned} \qquad \quad
\begin{aligned}
d_i s_j &= \begin{cases}
s_{j-1} d_i & \text{for $i < j$} \\
\mathrm{id} & \text{for $i = j, j+1$} \\
s_j d_{i-1} & \text{for $i > j+1$;}
\end{cases}
\end{aligned}
\end{equation*}
and in fact, a simplicial set may be completely specified by giving
its sets of $n$-simplices, together with face and degeneracy maps satisfying
the simplicial identities.

\begin{definition}\label{def:catalan}
  The \emph{Catalan simplicial set} $\mathbb C$ has its $n$-simplices
  given by \emph{Dyck words} of length $2n+2$; these are strings
  comprised of $(n+1)$ $U$'s and $(n+1)$ $D$'s such that the $i$th $U$
  precedes the $i$th $D$ for each $1 \leqslant i \leqslant n+1$.
  % \vskip0.5\baselineskip \item For each $\phi \colon [n] \to [m]$ in
  % $\Delta$, the map $\mathbb{C}_m \to \mathbb C_n$ acts on a Dyck word
  % $W$ of length $2m+2$ by replacing, for each $1 \leqslant k \leqslant
  % m+1$, the $k$th $U$ (respectively, the $k$th $D$) by
  % $\abs{\phi^{-1}(k-1)}$ consecutive $U$'s (respectively $D$'s).
  The face maps $d_i \colon \mathbb C_n \to \mathbb C_{n-1}$ act on a
  word $W$ by deleting the $i$th $U$ and $i$th $D$; the degeneracy maps
  $s_i \colon \mathbb C_{n-1} \to \mathbb C_n$ act on a word $W$ by
  repeating the $i$th $U$ and $i$th $D$.
\end{definition}
  
The sets of Dyck words of length $2n$ are a Catalan family of
sets---corresponding to (i) or (r) in Stanley's
enumeration~\cite{StanleyEC2}---and so we have that $\abs{\mathbb C_n}
= C_{n+1}$, the $(n+1)$st Catalan number.

\begin{remark}
  The sets of $n$-simplices of $\mathbb C$ are not quite a Catalan
  family, due to the dimension shift causing us to omit the $0$th
  Catalan number. We may rectify this by viewing $\mathbb C$ as an
  \emph{augmented} simplicial
  set. % The augmented simplicial category $\Delta_+$ is the
  % category of \emph{all} (not merely non-empty) finite ordinals and
  % order-preserving maps; it adjoins to $\Delta$ a new object $[-1]$,
  % the empty ordinal, and a new morphism $\delta_0 \colon [-1] \to
  % [0]$ satisfying $\delta_0 \delta_0 = \delta_1 \delta_0 \colon [-1]
  % \to [1]$.
  An augmented simplicial set is a presheaf on $\Delta_+$, the
  category of all finite ordinals and order-preserving maps; it is
  equally given by a simplicial set $X$ together with a set $X_{-1}$
  of $(-1)$-simplices and an ``augmentation'' map $d_0 \colon X_0 \to
  X_{-1}$ satisfying $d_0d_0 = d_0d_1 \colon X_1 \to X_{-1}$. By
  allowing $n$ to range over $\{-1\} \cup \mathbb N$ in the definition
  of the Catalan simplicial set $\mathbb C$, it becomes an augmented
  simplicial set with the property that its sets of $(n-1)$-simplices
  (for $n \in \mathbb N$) are a Catalan family.
\end{remark}

% We introduce some further useful notation for composites of face
% maps. Given $0 \leqslant a_0 < \dots < a_n \leqslant k$ and a
% $k$-simplex $x \in X_k$, we write $x_{a_0 \dots a_n}$ for the image of
% $x$ under the action of the injection $ [n] \to [k]$ in $\Delta$ whose
% image is $\{a_0, \dots, a_n\}$. With these conventions, we may notate
% low dimensional simplices $x \in X_1$ and $y \in X_2$ of a simplicial
% set as
%  \begin{equation*}
%    \cd{x_0 \ar[r]^{x} & x_1} \qquad \text{and} \qquad
%    \twosimp{y_0}{y_1}{y_2}{y_{01}}{y_{12}}{y_{02}}{y}
%  \end{equation*}
%  respectively.
In order to understand the Catalan simplicial set as a simplicial
set, we must understand the face and degeneracy relations between its
simplices. In low dimensions, we see directly that
$\mathbb C$ has:
   \begin{itemize}
   \item A unique $0$-simplex $UD$, which we write as $\star$;
   \item Two $1$-simplices $UUDD$ and $UDUD$, the first of which is
     $s_0(\star)$ and the second of which is non-degenerate; we write these as $e =
     s_0(\star) \colon \star \to \star$ and $c \colon \star \to
     \star$;
   \item Five $2$-simplices: three degenerate ones $UUUDDD$, $UUDDUD$ and
     $UDUUDD$, and two non-degenerate ones $UUDUDD$ and $UDUDUD$.
     We depict these, and their faces, by:
     \begin{equation}
     \begin{gathered}
       \twosimp{\star}{\star}{\star}{e}{e}{e}{\substack{\,s_0(e)\\=s_1(e)}} \qquad
       \twosimp{\star}{\star}{\star}{e}{c}{c}{s_0(c)} \qquad
       \twosimp{\star}{\star}{\star}{c}{e}{c}{s_1(c)} \\
       \twosimp{\star}{\star}{\star}{c}{c}{c}{t} \qquad
       \twosimp{\star}{\star}{\star}{e}{e}{c}{i} \rlap{\quad .}
     \end{gathered}\label{eq:1}
   \end{equation}
   \end{itemize}

 In higher dimensions, the simplices of $\mathbb C$ will be determined
 by \emph{coskeletality}. A simplicial set is called
 \emph{$r$-coskeletal} when every $n$-boundary with $n > r$ has a
 unique filler; here, an \emph{$n$-boundary} in a simplicial set is a
 collection of $(n-1)$-simplices $(x_0, \dots, x_n)$ satisfying
 $d_j(x_i) = d_i(x_{j+1})$ for all $0 \leqslant i \leqslant j < n$; a
 \emph{filler} for such a boundary is an $n$-simplex $x$ with $d_i(x)
 = x_i$ for $i = 0, \dots, n$.

 \begin{proposition}
   \label{prop:1}
 The Catalan simplicial set is $2$-coskeletal.  
 \end{proposition}
 \begin{proof}
   % To each $W \in \mathbb C_n$, we may associate a binary relation $R_W
   % \subset \{1, \dots, n+1\}^2$ by $R_W = \{(i, j) : \text{the $i$th
   %   $U$ precedes the $j$th $D$ in $W$}\}$.
   % To each Dyck word $W$ of length $2n+2$, we associate a relation 
   For each natural number $n$, let $\mathbb K_n$ be the set of binary
   relations $R \subset \{0, \dots, n\}^2$ such that
   \begin{enumerate}[(i)]
   \item $i \mathbin R j$ implies $i < j$;
   \item $i < j < k$ and $i \mathbin R k$ implies $i \mathbin R j$ and
     $j \mathbin R k$.
   \end{enumerate}
   For each $n \geqslant 0$, there is a bijection $\mathbb C_{n} \to
   \mathbb K_{n}$ which sends a Dyck word $W$ to the set of those
   pairs $i <j$ such that the $(j+1)$st $D$ precedes the $(i+1)$st $U$
   in $W$; these bijections induce a simplicial structure on the
   $\mathbb K_n$'s, and it suffices to prove that this induced
   structure is
   $2$-coskeletal.

   % We may view an $n$-simplex of $\mathbb K$ as a relation defined
   % on
   % any totally ordered $(n+1)$-element set; in this way,
   We may identify the faces of an $n$-simplex $R \in \mathbb K_n$
   with the restrictions of $R$ to the $(n+1)$ distinct $n$-element
   subsets of $\{0, \dots, n\}$. An arbitrary collection $(R_0, \dots,
   R_n)$ of such relations, seen as elements of $\mathbb K_{n-1}$,
   comprises an $n$-boundary just when each $R_i$ and $R_j$ agree on
   the intersections of their domains. In this situation, there is a a
   unique relation $R \subset \{0, \dots, n\}^2$ restricting back to
   the given $R_i$'s, and satisfying (i) since each $R_i$ does. If
   $n>2$, then each triple $0 \leqslant i < j < k \leqslant n$ will
   lie entirely inside the domain of some $R_\ell$, and so the
   relation $R$ will satisfy (ii) since each $R_\ell$ does, and thus
   constitute an element of $\mathbb K_n$. Thus for $n >
   2$, each $n$-boundary of $\mathbb K \cong \mathbb C$ has a unique filler.
% a typical face map $d_i
%    \colon \mathbb K_{n} \to \mathbb K_{n-1}$ acts on a relation $R$ by
%    removing $i$ from its domain and renumbering; thus $j
%    \mathbin{d_i(R)} k$ iff $\delta_i(j) \mathbin R \delta_i(k)$.
%    $\delta_k(i') = i$ and $\delta_k(j') = j$. It follows that any
%    simplex of $\mathbb K$ of dimension $\geqslant 1$ is uniquely
%    determined by its $1$-dimensional faces.
% It
%    follows that, for any $n > 2$, if we are given relations $R_0,
%    \dots, R_n \in \mathbb K_{n-1}$, then we may define a relation $R
%    \in \mathbb K_n$ by $i \mathbin R j$ iff there exists $0 \leqslant
%    k \leqslant n$ and $i', j'$ such that $i' \mathbin{R_k} j'$ and
%    $\delta_k(i') = i$ and $\delta_k(j') = j$.
 \end{proof}

 We now give  one further description of the Catalan
 simplicial set, perhaps the most appealing: we will exhibit it as the
 monoidal nerve of a particularly simple monoidal category, namely the
 poset $\mathbf 2 = \bot \leqslant \top$, seen as a monoidal category
 with tensor product given by disjunction.

 We first explain what we mean by this.  Recall that if $\CA$ is a
 category, then its \emph{nerve} $\mathrm N(\CA)$ is the simplicial
 set whose $0$-simplices are objects of $\CA$, and whose $n$-simplices
 for $n > 0$ are strings of $n$ composable morphisms. Since the face
 and degeneracy maps are obtained from identities and composition in
 $\CA$, the nerve in fact encodes the entire category structure of
 $\CA$.

 Suppose now that $\CA$ is a \emph{monoidal category} in the sense
 of~\cite{ML1963}---thus, equipped with a tensor product functor $\otimes
 \colon \CA \times \CA \to \CA$, a unit object $I \in \CA$, and
 families of natural isomorphisms $\alpha_{ABC} \colon (A \otimes B)
 \otimes C \cong A \otimes (B \otimes C)$, $\lambda_A \colon I \otimes
 A \cong A$ and $ \rho_A \colon A \cong A \otimes I$, satisfying
 certain coherence axioms which we recall in detail in
 Section~\ref{sec:higher} below. In this situation, the nerve of $\CA$
 as a category fails to encode any information concerning the monoidal
 structure. However, by viewing $\CA$ as a one-object bicategory
 (=weak $2$-category), we may form a different nerve which \emph{does}
 encode this extra information.
 \begin{definition}\label{def:monnerve}
   Let $\CA$ be a monoidal category. The \emph{monoidal nerve} of $\CA$ is the
   simplicial set $\mathrm N_\otimes(\CA)$ defined as follows:
\begin{itemize}
\item There is a unique $0$-simplex, denoted $\star$.
\item A $1$-simplex is an object $A \in \CA$; its two faces are
  necessarily $\star$.
\item A $2$-simplex is a morphism $A_{12} \otimes A_{01} \to A_{02}$
  in $\CA$; its three faces are $A_{12}$,
  $A_{02}$ and $A_{01}$.
\item A $3$-simplex is a commuting diagram
\def\xab{A_{01}}
\def\xbc{A_{12}}
\def\xcd{A_{23}}
\def\xac{A_{02}}
\def\xad{A_{03}}
\def\xbd{A_{13}}
\def\xbcd{A_{123}}
\def\xacd{A_{023}}
\def\xabd{A_{013}}
\def\xabc{A_{012}}
  \begin{equation}\label{eq:threesimpdiag}
\monnerve
\end{equation} in $\CA$; its four faces are $A_{123}, A_{023},
A_{013}$ and $A_{012}$.
\item Higher-dimensional simplices are determined by $3$-coskeletality.
\end{itemize}
The degeneracy of the unique $0$-simplex is the unit object $I \in  
\CA$; the two degeneracies $s_0(A), s_1(A)$ of a $1$-simplex are the
respective coherence constraints $\rho^{-1}_A \colon A \otimes I \to
A$ and $\lambda_A \colon I \otimes A \to A$; the three degeneracies of
a $2$-simplex are simply the assertions that certain diagrams commute,
which is so by the axioms for a monoidal category. Higher degeneracies
are determined by coskeletality.
\end{definition}

Note that, because the monoidal nerve arises from viewing a monoidal
category as a one-object bicategory, we have a dimension shift:
objects and morphisms of $\CA$ become $1$- and $2$-simplices of the
nerve, rather than $0$- and $1$-simplices. 

\begin{proposition}
  The simplicial set $\mathbb C$ is uniquely isomorphic to the
  monoidal nerve of the poset $\mathbf 2 = \mathord \bot \leqslant \mathord \top$, seen as a monoidal category
  under disjunction.
\end{proposition}
\begin{proof}
  In any monoidal nerve $\mathrm N_\otimes(\CA)$, each $3$-dimensional
  boundary has at most one filler, existing just when the
  diagram~\eqref{eq:threesimpdiag} associated to the boundary
  commutes. Since every diagram in a poset commutes, the nerve
  $\mathrm{N}_\otimes(\mathbf 2)$, like $\mathbb C$, is
  $2$-coskeletal. It remains to show that $\mathbb C \cong \mathrm
  N_\otimes(\mathbf 2)$ in dimensions $0,1,2$. In dimension $0$ this
  is trivial. In dimension $1$, any isomorphism must send $s_0(\star)
  = e \in \mathbb C_1$ to $s_0(\star) = \bot \in \mathrm
  N_\otimes(\mathbf 2)_1$ and hence must send $c$ to $\top$. In
  dimension $2$, the $2$-simplices of $\mathrm N_\otimes(\mathbf 2)$
  are of the form
 \begin{equation*}
   \twosimp{\star}{\star}{\star}{x_{01}}{x_{12}}{x_{02}}{}
 \end{equation*}  
 where $x_{12} \vee x_{01} \leqslant x_{02}$ in $\mathrm{N}_\otimes(\mathbf
 2)$. Thus in $\mathrm N_\otimes(\mathbf 2)$, as in $\mathbb C$, there is at
 most one $2$-simplex with a given boundary, and by examination of
 \eqref{eq:1}, we see that the same possibilities
 arise on both sides; thus there is a unique isomorphism $\mathbb C_2
 \cong \mathrm N_\otimes(\mathbf 2)_2$ compatible with the face maps, as required. 
\end{proof}

We conclude this section by investigating the non-degenerate simplices of the
Catalan simplicial set; these will be of importance in the following
sections, where they will play the role of basic coherence data in
higher-dimensional monoidal structures. We will see that these
non-degenerate simplices form a sequence of Motzkin sets. The
\emph{Motzkin numbers}~\cite{Donaghey1977Motzkin} $1,1,2,4,9,\dots$
are defined by the recurrence relations
\[
M_0 = 1 \qquad \text{and} \qquad M_{n+1} = M_n + \textstyle\sum_{k=0}^{n-1} M_{k}
M_{n-1-k}\rlap{ .}
\]
An $\mathbb N$-indexed family of sets is a \emph{sequence of Motzkin
  sets} if there are a Motzkin number of elements in each dimension.

\begin{example}
  A \emph{Motzkin word} is a string in the alphabet $\{U,C,D\}$ which,
  on striking out every $C$, becomes a Dyck word. The sets $\mathbb M_n$
  of Motzkin words of length $n$ are a sequence of Motzkin sets.\label{ex:2}
\end{example}

\begin{proposition}\label{prop:nd}
  The family $(\mathrm{nd}\,{\mathbb C}_n : n \in \mathbb N)$ of
  non-degenerate simplices of $\mathbb C$ is a sequence of Motzkin
  sets.
\end{proposition}
\begin{proof}
  It suffices to construct a bijection $\mathrm{nd}\,{\mathbb C}_n
  \cong \mathbb M_n$ for each $n$. In one direction, we have a map
  $\mathrm{nd}\,{\mathbb C}_n \to \mathbb M_n$ sending a
  non-degenerate Dyck word $W$ to the Motzkin word $M_1 \dots M_n$
  defined as follows: if the $i$th and $(i+1)$st $U$'s are adjacent in
  $W$, then $M_i = U$; if the $i$th and $(i+1)$st $D$'s are adjacent
  in $W$, then $M_i = D$; otherwise $M_i = C$. (Note that the first
  two cases are disjoint; a Dyck word $W$ satisfying both would have
  to be in the image of the $i$th degeneracy map).

  In the other direction, suppose given a Motzkin word $M = M_1 \dots
  M_n$. Let $a_1 < \dots < a_k$ enumerate all $i$ for which $M_i$ is
  $D$ or $C$, and let $b_1 < \dots < b_k$ enumerate all $i$ for which
  $M_i$ is $U$ or $C$. The inverse mapping $\mathbb M_n \to \mathrm{nd}\,{\mathbb C}_n
  $ now sends $M$ to the Dyck word
 %  For each $n$, consider the set $\mathbb M_{n}$ whose elements are
 %  sequences $(a_i, b_i)_{i=1}^k$ of pairs of natural numbers with the
 %  properties that:
 %   \begin{itemize}
 %   \item $1 \leqslant a_1 < \dots < a_k \leqslant n$;
 %   \item $1 \leqslant b_1 < \dots < b_k \leqslant n$;
 % \item $a_i \geqslant b_i$ for each $1 \leqslant i \leqslant k$.
 %   \end{itemize}
 %   This set has cardinality $C_{n+1}$ (c.f. item (l$^6$)
 %   in~\cite{StanleyAddendum}); an explicit isomorphism with $\mathbb
 %   D_{n+1}$ sends a Dyck word $W$ of length $2n + 2$ to the sequence
 %   of those pairs $(a,b)$ such that the $a+1$st $U$ immediately
 %   follows the $b$th $D$ in $W$, and in the converse direction, sends a
 %   sequence $(a_i, b_i)_{i=1}^k$ to the Dyck word
\[U^{a_1}D^{b_1}U^{a_2-a_1}D^{b_2-b_1} \cdots U^{a_k - a_{k-1}}D^{b_k - b_{k-1}}
U^{n+1-a_k} D^{n+1-b_k}\rlap{ .}\qedhere
\]
  % In (M4) of
  %  \cite[Section~3]{Donaghey1977Motzkin} is identified a subset
  %  $\mathbb P'_n \subset \mathbb P_n$ of size $M_{n}$, comprising
  %  those sequences $(a_i, b_i)_{i=1}^k$ in which each $1 \leqslant
  %  \ell \leqslant n$ appears either as some $a_i$ or as some $b_i$.
  % To complete the proof, it thus suffices to show that under the
  % isomorphism $\mathbb P_n \cong \mathbb D_{n+1}$, the subset $\mathbb
  % P'_n$ corresponds to the elements of $\mathbb D_{n+1}$ not in the
  % image of a degeneracy map.  Indeed, the subset $\mathbb P'_n$
  % corresponds to those $W \in \mathbb D_{n+1}$ such that, for each $1
  % \leqslant \ell \leqslant n$, either the $\ell+1$st $U$ is
  % immediately preceded by a $D$, or the $\ell$th $D$ is immediately
  % followed by a $U$.  The negation of this condition is that there
  % exists $1 \leqslant \ell \leqslant n$ such that the $\ell$th and
  % $\ell+1$st $U$'s are adjacent, and the $\ell$th and $\ell+1$st $D$'s
  % are adjacent; which, by Example~\ref{ex:dyck}, is the condition that
  % $W$ is in the image of the $\ell$th degeneracy map.
\end{proof}
Using this result, we may re-derive a well-known combinatorial
identity relating the Catalan and Motzkin numbers.
\begin{corollary}
  For each $n \geqslant 0$, we have $C_{n+1} = \sum_k {n \choose k} M_k$.
\end{corollary}
\begin{proof}
  Recall that the \emph{Eilenberg--Zilber lemma}~\cite[\S II.3]{GZ} states that
  every simplex $x \in X_n$ of a simplicial set $X$ is the image under
  a unique surjection $\phi \colon [n] \twoheadrightarrow [k]$ in
  $\Delta$ of a unique non-degenerate simplex $y \in X_k$. Since there
  are $n \choose k$ order-preserving surjections $[n]
  \twoheadrightarrow [k]$, 
  \[C_{n+1} = \abs{\mathbb C_n} = \textstyle\sum_{\phi \colon [n]
    \twoheadrightarrow [k]} \abs{\mathrm{nd}\,{\mathbb C}_k} = \sum_k
  {n \choose k} \abs{\mathrm{nd}\,{\mathbb C}_k} =  \sum_k {n \choose k} M_k \] as required.
\end{proof}
\renewcommand{\aa}{a} \renewcommand{\ll}{\ell} \newcommand{\rr}{r}
\newcommand{\kk}{k}

% \begin{remark}
%   With a little more work, we may use the proof of
%   Proposition~\ref{prop:nd} to give an explicit enumeration of the
%   non-degenerate simplices of the Catalan simplicial set.  

%  We
%   can establish an isomorphism $\mathbb M_n \cong \mathbb P'_n$ by
%   sending a Motzkin word $W$ to the sequence $(a_i, b_i)_{ i =1}^k$
%   wherein $(a_i)_{i=1}^k$ is the list of all positions of $W$ which
%   are $D$'s or $C$'s and $(b_i)_{i=1}^k$ is the list of all positions
%   which are $U$'s or $C$'s. By chaining the isomorphisms $\mathbb M_n
%   \cong \mathbb P'_n \cong \mathrm{nd}\ \mathbb D_{n+1} \cong
%   \mathrm{nd}\ \mathbb C_n$, we thus obtain an enumeration of the
%   non-degenerate $n$-simplices of $\mathbb C$ by the Motzkin words of
%   length $n$.
% \end{remark}

\section{First classifying properties}
\label{sec:class-prop-catal}
We now begin to investigate the \emph{classifying properties} of the
Catalan simplicial set, by looking at the structure picked out by maps
from $\mathbb C$ into the nerves of certain kinds of categorical
structure.

For our first classifying property, recall that a \emph{monoid} in a
monoidal category $\CA$ is given by an object $A \in \CA$ and
morphisms $\mu \colon A \otimes A \to A$ and $\eta \colon I \to A$
rendering commutative the three diagrams
    \begin{equation*}
      \def\xab{A}
      \def\xbc{A}
      \def\xcd{A}
      \def\xac{A}
      \def\xad{A}
      \def\xbd{A}
      \def\xbcd{\mu}
      \def\xacd{\mu}
      \def\xabd{\mu}
      \def\xabc{\mu}
     \monnerve 
\quad
\cd[@C-0.5em]{
A \ar[r]^-{\rho_A} \ar@{=}[d] & A \otimes I \ar[d]^{1 \otimes \eta} \\
A & A \otimes A \ar[l]^-{\mu}
} \quad  \cd{
I \otimes A \ar[r]^{\lambda_A} \ar[d]_{\eta \otimes 1} & A \\
A \otimes A \ar[ur]_{\mu}
} 
\end{equation*}

\begin{proposition}\label{prop:classifymonoids}
  If $\CA$ is a monoidal category, then to give a simplicial map $f
  \colon \mathbb C \to \mathrm N_\otimes(\CA)$ is equally to give a
  monoid in $\CA$.
\end{proposition}
\begin{proof}
  Since $\mathrm N_\otimes(\CA)$ is $3$-coskeletal, a simplicial map $f \colon
  \mathbb C \to \mathrm N_\otimes(\CA)$ is uniquely determined by where it
  sends non-degenerate simplices of dimension $\leqslant 3$. We have
  already described the non-degenerate simplices in dimensions
  $\leqslant 2$, while in dimension $3$, there are four such, given
  by
\begin{align*}
\aa &= (\tee,\tee,\tee,\tee) & 
\ll &= (\eye,s_1(c),\tee,s_1(c)) \\
\rr &= (s_0(c),\tee,s_0(c),\eye) &
\kk &= (\eye, s_1(c), s_0(c), \eye) \rlap{ .}
\end{align*}
Here, we take advantage of $2$-coskeletality of $\mathbb C$ to
identify a $3$-simplex $x$ with its tuple $(d_0(x), d_1(x), d_2(x),
d_3(x))$ of $2$-dimensional faces.  We thus see that to
give $f \colon \mathbb C \to \mathrm N_\otimes(\CA)$ is to give:
\begin{itemize}
  \item In dimension $0$, no data: $f$ must send $\star$ to $\star$;
  \item In dimension $1$, an object $A \in \CA$, the image of the
    non-degenerate simplex $c \in \mathbb C_1$;
  \item In dimension $2$, morphisms $\mu \colon A \otimes A \to A$ and
    $\eta' \colon I \otimes I \to A$, the images of the non-degenerate
    simplices $t, i \in \mathbb C_2$;
  \item In dimension $3$, commutative diagrams 
    \begin{equation*}
      \def\xab{A}
      \def\xbc{A}
      \def\xcd{A}
      \def\xac{A}
      \def\xad{A}
      \def\xbd{A}
      \def\xbcd{\mu}
      \def\xacd{\mu}
      \def\xabd{\mu}
      \def\xabc{\mu}
      f(a) = \monnerve
    \end{equation*}
   \begin{equation*}
      \def\xab{A}
      \def\xbc{I}
      \def\xcd{I}
      \def\xac{A}
      \def\xad{A}
      \def\xbd{A}
      \def\xbcd{\eta'}
      \def\xacd{\lambda_A}
      \def\xabd{\mu}
      \def\xabc{\lambda_A}
      f(\ell) = \monnerve
    \end{equation*}    \begin{equation*}
      \def\xab{I}
      \def\xbc{I}
      \def\xcd{A}
      \def\xac{A}
      \def\xad{A}
      \def\xbd{A}
      \def\xbcd{\rho^{-1}_A}
      \def\xacd{\mu}
      \def\xabd{\rho^{-1}_A}
      \def\xabc{\eta'}
      f(r) = \monnerve
    \end{equation*}
    \begin{equation*}
      \def\xab{I}
      \def\xbc{I}
      \def\xcd{I}
      \def\xac{A}
      \def\xad{A}
      \def\xbd{A}
      \def\xbcd{\eta'}
      \def\xacd{\lambda_A}
      \def\xabd{\rho^{-1}_A}
      \def\xabc{\eta'}
      f(k) = \monnerve
    \end{equation*}
    the images as displayed of the non-degenerate $3$-simplices
    of $\mathbb C$.
  \end{itemize}
  On defining $\eta = \eta' \circ \rho_A \colon I \to I \otimes I \to
  A$, we obtain a bijective correspondence between the data
  $(A,\mu,\eta')$ for a simplicial map $\mathbb C \to \mathrm N_\otimes(\CA)$
  and the data $(A,\mu,\eta)$ for a monoid in $\CA$. Under this
  correspondence, the axiom $f(a)$ for $(A, \mu, \eta')$ is clearly
  the same as the first monoid axiom for $(A,\mu,\eta)$; a short
  calculation with the axioms for a monoidal category shows that
  $f(\ell)$ and $f(r)$ correspond likewise with the second and third
  monoid axioms. This leaves only $f(k)$; but it is easy to show that
  this is automatically satisfied in any monoidal category. Thus
  monoids in $\CA$ correspond bijectively with simplicial maps
  $\mathbb C \to \mathrm N_\otimes(\CA)$ as claimed.
 \end{proof}

 \begin{remark}
   A generalisation of this classifying property concerns maps
   from $\mathbb C$ into the nerve of a \emph{bicategory} $\CB$ in the
   sense of~\cite{Ben1967}. Bicategories are ``many object'' versions
   of monoidal categories, and the nerve of a bicategory is a ``many
   object'' version of the monoidal nerve of
   Definition~\ref{def:monnerve}. An easy modification of the
   preceding argument shows that simplicial maps $\mathbb C \to
   \mathrm N(\CB)$ classify monads in the bicategory $\CB$.
 \end{remark}

\section{Higher classifying properties}\label{sec:higher}
The category $\mathrm{Cat}$ of small categories and functors bears a
monoidal structure given by cartesian product, and monoids with respect
to this are precisely small \emph{strict} monoidal categories---those
for which the associativity and unit constraints $\alpha$, $\lambda$
and $\rho$ are all identities.  It follows by
Proposition~\ref{prop:classifymonoids} that simplicial maps $\mathbb C
\to \mathrm N_\otimes(\mathrm{Cat})$ classify small strict monoidal
categories.
The purpose of this section is to show that, in fact, we may also classify both
\begin{enumerate}[(i)]
\item Not-necessarily-strict monoidal categories; and
\item \emph{Skew-monoidal} categories in the sense of~\cite{Szl2012}
\end{enumerate}
by simplicial maps from $\mathbb C$ into suitably modified nerves of
$\mathrm{Cat}$, where the modifications at issue involve changing the
simplices from dimension $3$ upwards. The $3$-simplices will no longer
be commutative diagrams as in~\eqref{eq:threesimpdiag}, but rather
diagrams commuting up to an natural transformation, invertible in the
case of (i) but not necessarily so for (ii). The $4$-simplices will
be, in both cases, suitably commuting diagrams of natural
transformations, while higher simplices will be determined by coskeletality
as before.
Note that, to obtain these new classification results, we do not need
to change $\mathbb C$ itself, only what we map it into. The
change is from something $3$-coskeletal to something $4$-coskeletal,
which means that the non-degenerate $4$-simplices of $\mathbb C$ come
into play. As we will see, these encode precisely the coherence
axioms for monoidal or skew-monoidal structure.

Before continuing, let us make precise the definition of skew-monoidal
category. As explained in the introduction, the notion was introduced
by Szlach\'anyi in~\cite{Szl2012} to describe structures arising in
quantum algebra, and generalises Mac~Lane's notion of monoidal
category by dropping the requirement that the coherence constraints be
invertible.
\begin{definition}
A \emph{skew-monoidal category} is a  
 category $\CA$ equipped with a unit element $I \in \CA$, a tensor
product $\otimes \colon \CA \times \CA \to \CA$, and natural families
of (not neccesarily invertible) constraint maps 
\begin{equation}
\begin{gathered}
  \alpha_{ABC} \colon (A \otimes B) \otimes C \to A \otimes (B \otimes
  C)\\
  \lambda_{A} \colon I \otimes A \to A \quad \text{and} \quad \rho_A \colon A \to A
  \otimes I 
\end{gathered}\label{eq:data}
\end{equation}
subject to the commutativity of the
following diagrams---wherein tensor is denoted by juxtaposition---for
all $A,B,C,D \in \CA$:

\begin{equation*}
\def\xa{((AB)C)D}
\def\xc{(AB)(CD)}
\def\xe{A(B(CD))}
\def\xb{\llap{$(A(B$}C))D}
\def\xd{A((B\rlap{$C)D)$}}
\def\xo{\textstyle (5.1)}
\def\xac{\alpha}
\def\xce{\alpha}
\def\xab{\alpha 1}
\def\xbd{\alpha}
\def\xde{1 \alpha}
\pent[0.1cm]
 \quad
 \cd[@!C@C-6pt]{
  (AI)B \ar[r]^{\alpha} & A(IB) \ar[d]^{1\lambda}
  \ar@{}[dl]|{\textstyle (5.2)}& \\
 AB \ar[u]^{\rho1} \ar[r]_{\mathrm{id}} & AB
 }
\end{equation*}
\begin{equation*}
\cd[@!C@C-22pt]{
& I(AB) \ar[dr]^{\lambda} \ar@{}[d]|(0.6){\textstyle (5.3)} & \\
(IA)B \ar[ur]^{\alpha} \ar[rr]_{\lambda1} &  & AB
}
\quad
\cd[@!C@C-20pt]{
& (AB)I \ar[dr]^{\alpha} \ar@{}[d]|(0.6){\textstyle (5.4)} & \\
AB \ar[ur]^{\rho} \ar[rr]_{1\rho} &  & A(BI)
}
\quad
\cd[@!C@C-3pt@R+4pt]{
& II \ar[dr]^{\lambda} \ar@{}[d]|(0.6){\textstyle (5.5)} & \\
I \ar[ur]^{\rho} \ar[rr]_{\mathrm{id}} && I\rlap{ .}
}
\end{equation*}
\end{definition}
A skew-monoidal category in which $\alpha$, $\lambda$ and $\rho$ are
invertible is exactly a monoidal category; the axioms (5.1)--(5.5) are
then Mac~Lane's original five axioms~\cite{ML1963}, justified by the
fact that they imply the commutativity of \emph{all} diagrams of
constraint maps. In the skew case, this justification no longer
applies, as the axioms no longer force every diagram of constraint
maps to commute; for example, we need not have $1_{I \otimes I} =
\rho_I \circ \lambda_I \colon I \otimes I \to I \otimes I$. The
classification of skew-monoidal structure by maps out of the Catalan
simplicial set can thus be seen as an alternative justification of the
axioms in the absence of such a result.

Before giving our classification result, we describe the modified
nerves of $\mathrm{Cat}$ which will be involved. The possibility of
taking natural transformations as $2$-cells makes $\mathrm{Cat}$ not
just a monoidal category, but a \emph{monoidal bicategory} in the
sense of~\cite{GPS}. Just as one can form a nerve of a monoidal
category by viewing it as a one-object bicategory, so one can form a
nerve of a monoidal bicategory by viewing it as a one-object
tricategory (=weak 3-category), and in fact, various nerve
constructions are possible---see~\cite{GRoT}. The following
definitions are specialisations of some of these nerves to the
case of $\mathrm{Cat}$.

\begin{definition}
  The \emph{lax nerve} $\mathrm{N}_\ell(\mathrm{Cat})$ of the monoidal bicategory
  $\mathrm{Cat}$ is the simplicial set defined as follows:
\begin{itemize}
\item There is a unique $0$-simplex, denoted $\star$.
\item A $1$-simplex is a (small) category $\CA_{01}$; its two faces are
  both $\star$.
\item A $2$-simplex is a functor $A_{012} \colon \CA_{12} \times \CA_{01} \to \CA_{02}$.
\item A $3$-simplex is a natural transformation
  \begin{equation*}
\def\xab{\CA_{01}}
\def\xbc{\CA_{12}}
\def\xcd{\CA_{23}}
\def\xac{\CA_{02}}
\def\xad{\CA_{03}}
\def\xbd{\CA_{13}}
\def\xbcd{A_{123}}
\def\xacd{A_{023}}
\def\xabd{A_{013}}
\def\xabc{A_{012}}
\def\xabcd{A_{0123}}
\laxnerve
\end{equation*} with $1$-cell components
\[
(A_{0123})_{a_{23}, a_{12}, a_{01}} \colon A_{013}(A_{123}(a_{23},a_{12}),a_{01}) \to
A_{023}(a_{23},A_{012}(a_{12},a_{01}))\rlap{ .}
\]
\item A $4$-simplex is a quintuple of appropriately-formed natural transformations
$(A_{1234}, A_{0234}, A_{0134}, A_{0124}, A_{0123})$ making the
pentagon
\begin{equation*}
  \cd[@C-9.6em@R-1.5em]{& \ \ \ \ \ \ \ \ \ \ \ \ \ \ \ \ \ \ &
    A_{024}(A_{234}(a_{34}, a_{23}), A_{012}(a_{12}, a_{01}))
    \ar[dddr]^{A_{0234}}\\ \\
    A_{014}(A_{124}(A_{234}(a_{34}, a_{23}), a_{12}), a_{01})
    \ar[uurr]^{A_{0124}} \ar[dd]_{A_{014}(A_{1234}, 1)}\\
    & & & 
    A_{034}(a_{34}, A_{023}(a_{23}, A_{012}(a_{12}, a_{01}))) 
    \\ 
    A_{014}(A_{134}(a_{34}, A_{123}(a_{23}, a_{12})), a_{01}) \ar[ddrr]_{A_{0134}}
    \\ \\
    &&A_{034}(a_{34}, A_{013}(A_{123}(a_{23}, a_{12}), a_{01}))
    \ar[uuur]_{A_{034}(1, A_{0123})}
  }
\end{equation*}
commute in $\CA_{04}$ for all $(a_{01}, a_{12}, a_{23}, a_{34}) \in
\CA_{01} \times \CA_{12} \times \CA_{23} \times \CA_{34}$.
\item Higher-dimensional simplices are determined by $4$-coskeletality, and
face and degeneracy maps are defined as before.
\end{itemize}
The \emph{pseudo nerve} $\mathrm{N}_p(\mathrm{Cat})$
is defined identically except that the natural transformations
occurring in dimensions $3$ and $4$ are required to be invertible.
\end{definition}
We are now ready to give our higher classifying property of the
Catalan simplicial set.
\begin{proposition}
  To give a simplicial map $f
  \colon \mathbb C \to \mathrm N_p(\mathrm{Cat})$ is equally to give a
  small monoidal category; to give a simplicial map $f
  \colon \mathbb C \to \mathrm N_\ell(\mathrm{Cat})$ is equally to
  give a small skew-monoidal category. 
\end{proposition}
\begin{proof}
 First we prove the second statement.  Since $\mathrm N_\ell(\mathrm{Cat})$
  is $4$-coskeletal, a simplicial map into it is uniquely determined
  by where it sends non-degenerate simplices of dimension $\leqslant
  4$. In dimensions $\leqslant 3$, to give $f \colon \mathbb C \to
  \mathrm N_\ell(\mathrm{Cat})$ is to give:
\begin{itemize}
  \item In dimension $0$, no data: $f$ must send $\star$ to $\star$;
  \item In dimension $1$, a small category $\CA = f(c)$;
  \item In dimension $2$, a functor $\otimes = f(t) \colon \CA \times \CA
    \to \CA$ and an object $I \in \CA$ picked out by the functor $f(i)
    \colon 1 \times 1 \to \CA$;
  \item In dimension $3$, natural transformations
    \begin{equation*}
      \def\xab{\CA}
      \def\xbc{\CA}
      \def\xcd{\CA}
      \def\xac{\CA}
      \def\xad{\CA}
      \def\xbd{\CA}
      \def\xbcd{\otimes}
      \def\xacd{\otimes}
      \def\xabd{\otimes}
      \def\xabc{\otimes}
      \def\xabcd{\alpha}
      f(a) = \laxnerve
    \end{equation*}
   \begin{equation*}
      \def\xab{\CA}
      \def\xbc{1}
      \def\xcd{1}
      \def\xac{\CA}
      \def\xad{\CA}
      \def\xbd{\CA}
      \def\xbcd{f(i)}
      \def\xacd{\cong}
      \def\xabd{\otimes}
      \def\xabc{\cong}
      \def\xabcd{\lambda}
      f(\ell) = \laxnerve
    \end{equation*}    \begin{equation*}
      \def\xab{1}
      \def\xbc{1}
      \def\xcd{\CA}
      \def\xac{\CA}
      \def\xad{\CA}
      \def\xbd{\CA}
      \def\xbcd{\cong}
      \def\xacd{\otimes}
      \def\xabd{\cong}
      \def\xabc{f(i)}
      \def\xabcd{\rho}
      f(r) = \laxnerve
    \end{equation*}
    \begin{equation*}
      \def\xab{1}
      \def\xbc{1}
      \def\xcd{1}
      \def\xac{\CA}
      \def\xad{\CA}
      \def\xbd{\CA}
      \def\xbcd{f(i)}
      \def\xacd{\cong}
      \def\xabd{\cong}
      \def\xabc{f(i)}
      \def\xabcd{\kappa}
      f(k) = \laxnerve
    \end{equation*}
    which are equally well natural families $\alpha$, $\lambda$ and
    $\rho$ as in~\eqref{eq:data} together with a map $\kappa_\star
    \colon I \to I$.
  \end{itemize}
  So the data in dimensions $\leqslant 3$ for a simplicial map
  $\mathbb C \to \mathrm N_\ell(\mathrm{Cat})$ is the data
  $(\CA,\otimes,I,\alpha,\lambda,\rho)$ for a small skew-monoidal
  category augmented with a map $\kappa_\star \colon I \to I$ in
  $\CA$. It remains to consider the action on non-degenerate
  $4$-simplices of $\mathbb C$. There are nine such, given by:
\begin{equation*}
\begin{tabular}{ l l l }
$A1 = (\aa,\aa,\aa,\aa,\aa)$ \hspace{0.7cm} &
$A6 = (s_0(\eye),\ll, \kk, \rr, s_2(\eye))$ \\
$A2 = (\rr,s_1(\tee),\aa,s_1(\tee),\ll)$ &
$A7 = (\kk,\ll, s_0s_1(c), \rr, \kk)$  \\
$A3 = (\ll,\ll,s_2(\tee),\aa,s_2(\tee))$ &
$A8 = (\rr, s_1(\tee), s_0(\tee), \rr, \kk)$ \\
$A4 = (s_0(\tee),\aa,s_0(\tee),\rr,\rr)$ \hspace{0.7cm} &
$A9 = (\kk, \ll, s_2(\tee), s_1(\tee), \ll)$ \\
$A5 = (s_1(\eye),s_2(\eye), \kk, s_0(\eye), s_1(\eye))$ &
\end{tabular}
\end{equation*}
where as before, we take advantage of coskeletality of $\mathbb C$ to
identify a $4$-simplex with its tuple of $3$-dimensional faces.
The images of these simplices each assert the commutativity of a
pentagon of natural transformations involving $\alpha$, $\rho$,
$\lambda$ or $\kappa$; explicitly, they assert that for any $A,B,C,D \in \CA$,
the following pentagons commute in $\CA$:
\begin{equation*}
\!\!\!
\def\xa{((AB)C)D}
\def\xc{(AB)(CD)}
\def\xe{A(B(CD))}
\def\xb{\llap{$(A(B$}C))D}
\def\xd{A((B\rlap{$C)D)$}}
\def\xo{\textstyle\text{(A1)}}
\def\xac{\alpha}
\def\xce{\alpha}
\def\xab{\alpha 1}
\def\xbd{\alpha}
\def\xde{1 \alpha}
\pent[0.1cm]
\ 
\def\xa{AB}
\def\xc{AB}
\def\xe{AB}
\def\xb{(AI)B}
\def\xd{A(IB)}
\def\xo{\textstyle\text{(A2)}}
\def\xac{1}
\def\xce{1}
\def\xab{\rho 1}
\def\xbd{\alpha}
\def\xde{1 \lambda}
\pent[0.1cm]
\ 
\def\xa{(IA)B}
\def\xc{I(AB)}
\def\xe{AB}
\def\xb{AB}
\def\xd{AB}
\def\xo{\textstyle\text{(A3)}}
\def\xac{\alpha}
\def\xce{\lambda}
\def\xab{\lambda 1}
\def\xbd{1}
\def\xde{1}
\pent[0.1cm]
\end{equation*}
\begin{equation*}
\def\xa{AB}
\def\xc{(AB)I}
\def\xe{A(BI)}
\def\xb{AB}
\def\xd{AB}
\def\xo{\textstyle\text{(A4)}}
\def\xac{\rho}
\def\xce{\alpha}
\def\xab{1}
\def\xbd{1}
\def\xde{1\rho}
\pent[0.1cm]\ 
\def\xa{I}
\def\xc{I}
\def\xe{I}
\def\xb{I}
\def\xd{I}
\def\xo{\textstyle\text{(A5)}}
\def\xac{1}
\def\xce{1}
\def\xab{1}
\def\xbd{\kappa_\star}
\def\xde{1}
\pent[0.1cm]\ 
\def\xa{I}
\def\xc{II}
\def\xe{I}
\def\xb{I}
\def\xd{I}
\def\xo{\textstyle\text{(A6)}}
\def\xac{\rho_I}
\def\xce{\lambda_I}
\def\xab{1}
\def\xbd{\kappa_\star}
\def\xde{1}
\pent[0.1cm]
\end{equation*}
\begin{equation*}
\def\xa{I}
\def\xc{II}
\def\xe{I}
\def\xb{I}
\def\xd{I}
\def\xo{\textstyle\text{(A7)}}
\def\xac{\rho_I}
\def\xce{\lambda_I}
\def\xab{\kappa_\star}
\def\xbd{1}
\def\xde{\kappa_\star}
\pent[0.1cm]\ 
\def\xa{A}
\def\xc{AI}
\def\xe{AI}
\def\xb{AI}
\def\xd{AI}
\def\xo{\textstyle\text{(A8)}}
\def\xac{\rho}
\def\xce{1}
\def\xab{\rho}
\def\xbd{1}
\def\xde{1\kappa}
\pent[0.1cm]\ 
\def\xa{IA}
\def\xc{IA}
\def\xe{A\rlap{ .}}
\def\xb{IA}
\def\xd{IA}
\def\xo{\textstyle\text{(A9)}}
\def\xac{1}
\def\xce{\lambda}
\def\xab{\kappa 1}
\def\xbd{1}
\def\xde{\lambda}
\pent[0.1cm]
\end{equation*}
Note first that (A5) forces $\kappa_\star = 1_I \colon I \to I$. Now
(A1)--(A4) express the axioms (5.1)--(5.4), both (A6) and (A7) express
axiom (5.5), whilst (A8) and (A9) are trivially satisfied. Thus the
$4$-simplex data of a simplicial map $\mathbb C \to
\mathrm{N}_\ell(\mathrm{Cat})$ exactly express the skew-monoidal
axioms and the fact that the additional datum $\kappa_\star \colon I
\to I$ is trivial; whence a simplicial map $\mathbb C \to
\mathrm{N}_\ell(\mathrm{Cat})$ is precisely a small skew-monoidal
category.

The same proof now shows that a simplicial map $\mathbb C \to
\mathrm{N}_p(\mathrm{Cat})$ is precisely a small monoidal category,
under the identification of monoidal categories with skew-monoidal
categories whose constraint maps are invertible.
\end{proof}

\end{document}